\documentclass[reqno, 10pt]{amsart}
\usepackage{mathrsfs}
\usepackage{amssymb}
\usepackage{latexsym}
\usepackage{yfonts}
\usepackage{amssymb,amsmath,euscript}
\usepackage{color}
\usepackage[toc,page]{appendix}

\usepackage{amssymb,amsmath,amsthm,mathdesign}
\usepackage{pdfsync}
\usepackage{color}

\usepackage{graphicx,color}

\usepackage{amssymb,amsmath,amsthm,mathdesign}
\usepackage{pdfsync}
\usepackage{color}
\newtheorem{tm}{Theorem}[section]
\newtheorem{prop}[tm]{Proposition}
\newtheorem{defin}[tm]{Definition} 
\newtheorem{coro}[tm]{Corollary}
\newtheorem{lem}[tm]{Lemma}

\numberwithin{equation}{section}

\newcommand{\RR}[1]{\mathbb{#1}}

\newcommand{\rd}{{\mathbb R^d}}

\newcommand{\rr}{{\mathbb R}}

\def\R{{\mathbb R}}
\def\N{{\mathbb N}}

\def\a{\alpha}

\def\E{{\mathbb E}}
\def\P{{\mathbb P}}

\allowdisplaybreaks

\begin{document}

\title[Intermittency fronts  ]{Intermittency fronts for space-time fractional stochastic partial differential equations in $(d+1)$ dimensions}

\author{Sunday A. Asogwa}
\address{Department of Mathematics and Statistics, Auburn University, Auburn, AL 36849, USA}
\email{saa0020@auburn.edu}

\author{ Erkan Nane}
\address{Department of Mathematics and Statistics, Auburn University, Auburn, AL 36849, USA}
\email{nane@auburn.edu}


\date{\today}

\subjclass[2000]{}

\begin{abstract}
We consider  time fractional  stochastic heat type equation
$$\partial^\beta_tu_t(x)=-\nu(-\Delta)^{\alpha/2} u_t(x)+I^{1-\beta}_t[\sigma(u)\stackrel{\cdot}{W}(t,x)]$$ in $(d+1)$ dimensions, where  $\nu>0$, $\beta\in (0,1)$, $\alpha\in (0,2]$, $d<\min\{2,\beta^{-1}\}\a$,  $\partial^\beta_t$ is the Caputo fractional derivative, $-(-\Delta)^{\alpha/2} $ is the generator of an isotropic stable process,  $\stackrel{\cdot}{W}(t,x)$ is space-time white noise, and $\sigma:\R \to\RR{R}$ is Lipschitz continuous.  Mijena and  Nane proved in \cite{JebesaAndNane1} that : (i) absolute moments of the solutions of this equation grows exponentially; and (ii) the distances to the origin of the farthest high peaks of those moments grow exactly linearly with time. The last result was proved under the assumptions $\alpha=2$ and $d=1.$    In this paper we extend this result to the case $\alpha=2$ and $d\in\{1,2,3\}.$

\end{abstract}


\maketitle

\section{Introduction}
Recently time-fractional  diffusion equations were studied by researchers in many  applied and theoretical fields of science and engineering. A typical form of the time fractional diffusion equations is $\partial^\beta_tu=\nu \Delta u$  with $\beta\in (0,1)$.
These equations are related with anomalous diffusions or diffusions in non-homogeneous media, with random fractal structures; see, for instance, \cite{meerschaert-nane-xiao}.
The Caputo fractional derivative  defined first by Caputo \cite{Caputo} denoted by $\partial^\beta_t$ is defined for $0<\beta<1$ by
\begin{equation}\label{CaputoDef}
\partial^\beta_t u_t(x)=\frac{1}{\Gamma(1-\beta)}\int_0^t \frac{\partial
u_r(x)}{\partial r}\frac{dr}{(t-r)^\beta} .
\end{equation}
Its Laplace transform is
\begin{equation}\label{CaputolT}
\int_0^\infty e^{-st} \partial^\beta_t u_t(x)\,dt=s^\beta \tilde u_s(x)-s^{\beta-1} u_0(x),
\end{equation}
where $\tilde u_s(x) = \int_0^\infty e^{-st}u_t(x)dt$ and incorporates the initial value in the same way as the first
derivative.

For some deep and rigorous mathematical approaches  to time fractional diffusion (heat type) equations see \cite{Koc89,NANERW, Nig86, Wyss86}. The stochastic solutions to fractional diffusion equations can be realized through time-change by inverse stable subordinators and therefore we obtain time-changed  processes.   A couple of recent works in this field are \cite{mnv-09, meerschaert-nane-xiao, meerschaert-straka, OB09}.

\noindent Let $\gamma>0$, define the fractional integral by
$$I^{\gamma}_tf(t):=\frac{1}{\Gamma(\gamma)} \int _0^t(t-\tau)^{\gamma-1}f(\tau)d\tau.$$
For every  $\gamma>0$, and   $g\in L^\infty(\rr_+)$ or $g\in C(\rr_+)$, we have the following relation
$$ \partial _t^\gamma I^\gamma_t g(t)=g(t).$$

Mijena and Nane \cite{nane-et-al-2014} have given an argument using the time fractional Duhamel's principle to obtain the following equation:
 \begin{equation}\label{tfspde}
\begin{split}
 &\partial^\beta_tu_t(x)=-\nu(-\Delta)^{\alpha/2} u_t(x)+I^{1-\beta}_t[\sigma(u)\stackrel{\cdot}{W}(t,x)],\ \ t> 0,\, x\in\R^d;\\
 &u_t(x)|_{t=0}=u_0(x),
 \end{split}
 \end{equation}
where the initial datum $u_0$ is $L^p(\Omega)$-bounded ($p\ge 2$), that is,
\begin{equation}\label{Eq:Initial}
\sup_{x\in\R^d}\E[|u_0(x)|^p]<\infty,
\end{equation}
 $-(-\Delta)^{\alpha/2} $ is the fractional Laplacian with $\alpha\in (0,2]$, and $\stackrel{\cdot}{W}(t,x)$ is a  space-time white noise with $x\in \R^d$, modeling the random effects.
The fractional integral above in equation \eqref{tfspde} when $\sigma(u)=1$ for  functions $\phi \in L^2(\R^d)$   is defined as
 $$\int_{\R^d} \phi(x)I^{1-\beta}_t[\stackrel{\cdot}{W}(t,x)]dx= \frac{1}{\Gamma(1-\beta)} \int_{\R^d}\int _0^t(t-\tau)^{-\beta}\phi(x) W(d\tau, dx)$$
it is well defined only when $0<\beta<1/2$. It is a type of Rieman-Liouville process.

 It would be nice to consider the equation \eqref{tfspde} with the space-time white noise without the fractional integral.  For  related time fractional stochastic  equations  with different noise terms see \cite{Le-chen1, Le-chen2, Le-chen3, hu-hu-15}.


The noise $\stackrel{\cdot}{W}(t,x)$ is a space-time white noise  with $x\in \R^d$, which is assumed to be adapted with respect to a filtered probability space $(\Omega, \mathcal{F},  \mathcal{F}_t, \P)$, where $ \mathcal{F}$ is complete and the filtration $\{\mathcal{F}_t, t\geq 0\}$ is right continuous.

Let  $G_t(x)$  denote the heat kernel of the time fractional heat type equation
\begin{equation}\label{Eq:Green0}
\partial^\beta_t G_t(x)=-\nu(-\Delta)^{\alpha/2}G_t(x).
\end{equation}

The existence and uniqueness of the solution to (\ref{tfspde}) has been studied by  Mijena and Nane  \cite{nane-et-al-2014} under global Lipchitz conditions on $\sigma$, using the white noise  approach of  Walsh \cite{walsh}:
We say that an $\mathcal{F}_t$-adapted random field $\{u(t,x),\,t\ge 0,\,x\in\R^d\}$ is said to be a mild solution of (\ref{tfspde}) with initial value $u_0$ if the following integral equation is fulfilled
\begin{equation}\label{Eq:Mild}
u_t(x)=\int_{\R^d}u_0(y)G_t(x-y)dy+\int_0^t\int_{\R^d}\sigma(u_r(y))G_{t-r}(x-y)W(drdy).
\end{equation}

Let $T$ be a fixed positive number, and let $B_{T,p}$ denote the family of all $\mathcal{F}_t$-adapted random fields $\{u_t(x),\,t\in [0,\,T],\,x\in\R^d\}$ satisfying
\begin{equation}\label{Eq:SquareInt}
\sup_{x\in\R^d}\sup_{t\in[0,\,T]}\E\left[|u_t(x)|^p\right]<\infty,
\end{equation}
with the convention that $B_{T,2}=B_T$. It is easy to check that for each fixed $T$ and $p,$ $B_{T,p}$ is a Banach space.

Mijena and Nane \cite{nane-et-al-2014} proved the existence and uniqueness result  for the equation \eqref{tfspde} when
 $d<\min\{2,\beta^{-1}\}\a$:  equation (\ref{tfspde}) subject to (\ref{Eq:Initial}) and global Lipschitz condition on $\sigma$ has an a.s.-unique solution $u_t(x)$ that satisfies that for all $T>0,$
$u_t(x)\in B_{T,p}.$

 A related  time-fractional SPDE was studied by Chen  et al. \cite{chen-kim-kim-2014}. They have proved existence, uniqueness and regularity of the solutions to the time-fractional parabolic type SPDEs using cylindrical Brownian motion in Banach spaces, in line with the methods in \cite{daPrato-Zabczyk}. For a comparison of the two approaches to SPDE's see the paper by Dalang and Quer-Sardanyons \cite{Dalang-Quer-Sardanyons}.

In this paper we study intermittency fronts for the solution of the stochastic equation \eqref{tfspde}.
 We adopt the definition given in \cite[Chapter 7]{khoshnevisan-cbms}:
 The random field $u_t(x)$
is called  intermittent if $\inf_{z\in \rd}|\sigma(z)|>0$, and
$\eta_k(x)/k$ is strictly increasing for $k\geq 2$ for all $x\in \rd$, where
\begin{equation}\label{lyapunov-x}
\eta_k(x):=\liminf_{t\to\infty}\frac1t\log\E(|u_t(x)|^k).
\end{equation}

The following observation of Carmona and Molchanov \cite[Theorem 3.1.2]{carmona-molchanov} gives a sufficient condition for intermittency: see \cite[Proposition 7.2]{khoshnevisan-cbms} for a proof of the next proposition.

  \begin{prop}
  \label{intermittence-sufficient}
  If $\eta(k)<\infty$ for all sufficiently  large $k$, then the function $\eta$ is well-defined and convex on $(0,\infty)$. Moreover, If $\eta(k_0)>0$ for some $k_0>1$, then $k\to k^{-1}\eta(k) $ is strictly increasing on $[k_0,\infty)$
  \end{prop}

\begin{tm}[\cite{JebesaAndNane1}]\label{intermittency-thm}
Let  $d<\min\{2,\beta^{-1}\}\a$.
If $\inf_{z\in \rd}|u_0(z)|>0$, then $$\inf_{x\in \rd}\eta_2(x)\geq [C^\ast(L_\sigma)^2\Gamma(1-\beta d/\a)]^{\frac{1}{(1-\beta d/\a)}}$$
where
\begin{equation}\label{conecondition}
L_\sigma:=\inf_{z\in \rd}|\sigma(z)/z|.
\end{equation}
Therefore, the solution  $u_t(x)$ of \eqref{tfspde} is weakly intermittent when $\inf_{z\in \rd}|u_0(z)|>0$ and $L_\sigma>0$.
\end{tm}

 There is a huge literature on the study of intermittency of SPDEs, see, for example, \cite{foondun-khoshnevisan-09, khoshnevisan-cbms} and the reference therein.

According to the previous theorem the solution develops tall peaks over time which means that $t\to \sup_{x\in \R^d} \E{|u_t(x)|^2}$ grows exponentially rapidly with $t$. There appears another phenomena called intermittency fronts that the distances of the farthest peaks of the moments of the solution to \eqref{tfspde} grow  linearly with time as $\theta t$: if $\theta$ is sufficiently small, then the quantity $\sup_{|x|>\theta t} \E{|u_t(x)|^2}$ grows exponentially quickly as $t\to\infty$; whereas the preceding quantity vanishes exponentially rapidly if $\theta$ is sufficiently large.
In this work, we consider for every $\theta\geq 0$,
\begin{equation}\label{intermittecy-front-L}
\mathscr{L}(\theta):= \limsup_{t\rightarrow \infty}\frac{1}{t}\sup_{|x|>\theta t}\log \E\left(|u_t(x)|^2\right).
\end{equation}
We can think of $\theta_L>0$ as an \textit{intermittency lower front} if $\mathscr{L}(\theta)<0$ for all $\theta > \theta_L,$ and of $\theta_U>0$ as an \textit{intermittency upper front} if $\mathscr{L}(\theta)>0$ whenever $\theta<\theta_U.$

The following is our main theorem which establishes  bounds for $\theta_L$ and $\theta_U$ that extend the results of \cite{conus-khoshnevisan} and \cite{JebesaAndNane1} to the case of $\alpha=2$ and $d\in\{1,2,3\}$ for  time fractional SPDEs with crucial nontrivial changes to the methods in \cite{conus-khoshnevisan, khoshnevisan-cbms}.

\begin{tm}\label{maintheoremtwo}
Suppose that $d< \min\{2,\beta^{-1}\}\alpha, \ \alpha=2,$ \ and  measurable initial function $u_0:\RR{R}^{d}\to\RR{R}_+$  is bounded, has compact support, and is strictly positive in an open subset of $(0,\infty)^{d}$, and  $\sigma$  satisfies
 $\sigma(0)=0$. Then  the time fractional stochastic heat equation \eqref{tfspde} has a positive intermittency lower front. In fact,
\begin{equation}\label{intermittencypart1}
\mathscr{L}(\theta) < 0\ \ \mbox{if}\ \theta > (2\nu)^{1/\beta}(\mbox{Lip}_\sigma c_0)^{2(\frac{2-\beta}{2-\beta d})}.
\end{equation}
In addition, under the cone condition $L_\sigma>0-$where $L_\sigma$ was defined in \eqref{conecondition}-there exists $\theta_0>0$ such that
\begin{equation}\label{intermittencypart2}
\mathscr{L}(\theta) > 0\ \ \mbox{if}\ \theta \in (0,\theta_0).
\end{equation}
That is, in this case, the stochastic heat equation has a finite intermittency upper front.
\end{tm}
This theorem in the case of $d=1$  was proved by Mijena and  Nane \cite{JebesaAndNane1}.
In   the parabolic Anderson model  which is the  stochastic heat equation \eqref{tfspde},  when $\beta=1$ and $\sigma(x)=cx$,  it is now known that there exists a sharp intermittency front, namely $\theta_L = \theta_U$, see the work of Chen and Dalang \cite{chen-dalang}. It would be nice to consider equality of  $\theta_L = \theta_U$ for \eqref{tfspde} when $\beta\in (0,1)$. We will carry out this project in a forthcoming paper.

Next we want to give an outline of the paper. In section 2, we  recall some preliminary results on the subject from the literature. Hence proofs of the results here can be found in the literature, in particular see references therein. Next, we established some useful results that we used in the proof of our main result. Section 3 contains our main result, Theorem \ref{maintheoremtwo}.



\section{Preliminaries}
In this section we give some results about the heat kernel $G_t(x)$  of the time fractional heat type equation \eqref{Eq:Green0},
 and mention some basic facts about the integral (mild) solution of       \eqref{tfspde} in the sense of Walsh \cite{walsh}.
We know that $G_t(x)$ is the density function of $X(E_t)$, where $X$ is an isotropic $\a$-stable L\'evy process in $\R^d$ and $E_t=\inf \{u:\ D(u)>t\}$,  is the first passage time of a $\beta$-stable subordinator $D=\{D_r,\,r\ge0\}$, or the inverse stable subordinator of index $\beta$: see, for example, Bertoin \cite{bertoin} for properties of these processes, Baeumer and Meerschaert \cite{fracCauchy} for more on time fractional diffusion equations, and Meerschaert and Scheffler  \cite{limitCTRW} for properties of the inverse stable subordinator $E_t$.

Let $p_{{X(s)}}(x)$ and $f_{E_t}(s)$ be the density of $X(s)$ and $E_t$, respectively. Then the Fourier transform of $p_{{X(s)}}(x)$ is given by
\begin{equation}\label{Eq:F_pX}
\int _\rd e^{i\xi \cdot x}p_{X(s)}(x)dx=e^{-s\nu|\xi|^\a},
\end{equation}
and
\begin{equation}\label{Etdens0}
f_{E_t}(s)=t\beta^{-1}s^{-1-1/\beta}g_\beta(ts^{-1/\beta}),
\end{equation}
where $g_\beta(\cdot)$ is the density function of $D_1.$ The function $g_\beta(u)$ [cf. Meerschaert and Straka (2013)] is infinitely differentiable on the entire real line, with $g_\beta(s)=0$ for $s\le 0$. 

By using \eqref{Etdens0} and change of variable we can show that

\begin{equation}\label{Moment for g}
\mathbb{E}(D_{1}^{-\beta k})=\E(E_1^k)=\int_{0}^{\infty}w^{-\beta k}g_{\beta}(w)dw
\end{equation}

By conditioning, we have
\begin{equation}\label{Eq:Green1}
G_t(x)=\int_{0}^\infty p_{_{X(s)}}(x) f_{E_t}(s)ds.
\end{equation}

\begin{lem}[Lemma 2.1 in \cite{nane-et-al-2014}]\label{Lem:Green1} For $d < 2\a,$
\begin{equation}\label{Eq:Greenint}
\int_{{\R^d}}G^2_t(x)dx  =C^\ast t^{-\beta d/\a}
\end{equation}
where $C^\ast = \frac{(\nu )^{-d/\a}2\pi^{d/2}}{\a\Gamma(\frac d2)}\frac{1}{(2\pi)^d}\int_0^\infty z^{d/\a-1} (E_\beta(-z))^2 dz.$
\end{lem}


\begin{lem}[Lemma 2.2 in \cite{JebesaAndNane1}]\label{laplace-G} For $\lambda\in\rd$ and $\a = 2$,
$$\int_{\rd}e^{\lambda\cdot x}G_s(x)dx = E_\beta(\nu |\lambda|^2 s^\beta).$$
\end{lem}

 We barrow the following definition from \cite{foondun-khoshnevisan-09}: let $\Phi$ be a random field, and for every $\gamma>0$ and $k\in [2, \infty)$ define
\begin{equation}
\mathcal{N}_{\gamma,k}(\Phi) := \sup_{t\geq 0}\sup_{x\in\R^d}\left(e^{-\gamma t}||\Phi_t(x)||_k\right):= \sup_{t\geq 0}\sup_{x\in\R^d}\left(e^{-\gamma t}\bigg[\E|\Phi_t(x)|^k\bigg]^{1/k}\right).
\end{equation}
If we identify a.s.-equal random fields, then every $\mathcal{N}_{\gamma, k} $ becomes a norm. Moreover, $\mathcal{N}_{\gamma, k} $ and $\mathcal{N}_{\gamma', k} $ are equivalent norms for all $\gamma, \gamma'>0$ and $k\in [2,\infty).$ Finally, we note that if $\mathcal{N}_{\gamma, k}(\Phi) <\infty$ for some $\gamma >0$ and $k\in[2,\infty)$, then $\mathcal{N}_{\gamma, 2}{(\Phi)}<\infty$  as well, thanks to Jensen's inequality.

\begin{defin}
We denote by $\mathcal{L}^{\gamma, 2}$  the completion of the space of all simple  random fields in the norm $\mathcal{N}_{\gamma, 2}.$
\end{defin}

{
We next recall  the {\bf Walsh-Dalang Integral } briefly:
We use the Brownian filtration $\{\mathcal{F}_t\}$ and the Walsh-Dalang integrals as follows
\begin{itemize}
\item $(t,x)\to \Phi_t(x)$ is an elementary random field when $\exists 0\leq a<b$ and an $\mathcal{F}_a$-measurable $X\in L^2(\Omega)$ and $\phi\in L^2(\rd)$ such that
    $$
    \Phi_t(x)=X1_{[a,b]}(t)\phi(x)\ \ \ (t>0, x\in \rd).
    $$
 \item  If $h=h_t(x)$ is non-random and $\Phi$ is elementary, then
 $$
 \int h\Phi d\xi:=X\int_{(a,b)\times\rd}h_t(x)\phi(x)\xi(dtdx).
 $$
 \item The stochastic integral is Wiener's; well defined iff $h_t(x)\phi(x)\in L^2([a,b]\times\rd)$.
 \item We have Walsh isometry,
 $$
 \E \bigg(\bigg|\int h\Phi d\xi\bigg|^2\bigg)=\int_0^\infty\int_{\rd}dy[h_s(y)]^2\E(|\Phi_s(y)|^2).
 $$

\end{itemize}

}

Given a random field $\Phi := \{\Phi_t(x)\}_{t\geq 0, x\in\R^d}$ and space-time noise $\dot{W}$, we define the [space-time] \textit{stochastic convolution} $G\circledast \Phi$ to be the random field that is defined as
$$(G\circledast \Phi)_t(x) := \int_{(0,t)\times \R^d} G_{t-s}(y-x)\Phi_s(y)W(dsdy),$$
for $t>0$ and $x\in\R^d,$ and $(G\circledast \Phi)_0(x) := 0.$

Define
\begin{equation}
G_s^{(t,x)}(y) := G_{t-s}(y-x)\cdot {\bf{1}}_{(0,t)}(s)\ \ \ \mbox{for all}\ s\geq 0\ \mbox{and}\ y\in\R^d.
\end{equation}
Clearly, $G^{(t,x)}\in L^2(\R_+\times \R^d)$ for $\beta d/\a<1$; in fact,
$$\int_0^\infty ds \int_{\R^d} [G_s^{(t,x)}(y)]^2dy =  \int_0^t ds \int_{\R^d} [G_s(y)]^2dy = C^\ast t^{1-\beta d/\a}<\infty.$$
This computation follows from Lemma \ref{Lem:Green1}. Thus, we may interpret the random variable $(G\circledast \Phi)_t(x)$ as the stochastic integral $\int G_s^{(t,x)}\Phi dW$, provided that $\Phi$ is in $\mathcal{L}^{\gamma, 2}$ for some $\gamma>0.$

\subsection{Some Useful Lemmas}
We start this subsection with a very important and non trivial result. The next Lemma provide an $"a-priori"$ estimate which allows us to overcome some difficulties in the proof of the main result.
\begin{lem}\label{mylemma1}
For $\beta \in (0, \ 1)$,   $k\in\N\cup\{0\}$ and $ d\in\{1, 2,3\}$  define
\[
a_{k}^{d}(\beta):=\mathbb{E}(D_{1}^{-\beta(k-\frac{d}{4})})=\int_{0}^{\infty}w^{-\beta(k-\frac{d}{4})}g_{\beta}(w)dw.
\]
 Then
\begin{equation}\label{mylemmaEq1}
0<\  a_{k}^{d}(\beta)\ \leq 3\frac{\Gamma(1+k)}{\Gamma(1+\beta k)}, \quad \text{for}\quad k\geq1.
\end{equation}
\end{lem}
\begin{proof}
 First observe that
\begin{eqnarray*}
a_{k}^{d}(\beta)   & = & \int_{0}^{\infty} \frac{ g_{\beta}(w) }{ w^{ \beta(k-\frac{d}{4}) }}dw\nonumber\\
&=&\int_{0}^{1}\frac{ g_{\beta}(w)}{ w^{\beta(k-\frac{d}{4})}}dw+\int_{1}^{\infty}\frac{g_{\beta}(w)}{w^{\beta(k-\frac{d}{4})}}dw \nonumber\\
\end{eqnarray*}
Since for every $k\geq 1$ we  have
$$
\frac{1}{w^{\beta(k-\frac{d}{4})}}\leq \begin{cases}
                                        \frac{1}{w^{\beta k}} \quad \text{if}\ 0<w<1\\
                                        \frac{1}{w^{\beta( k-1)}} \quad \text{if}\ w> 1.
                                       \end{cases}
$$
 Using the uniqueness of Laplace Transform and Remark 3.1 in \cite{limitCTRW}, we can easily show  that $\E(D_{1}^{-\beta k}))=\E(E_t^k)=\frac{\Gamma(1+k)}{\Gamma(1+\beta k)}.$ Then we have that
\begin{eqnarray}\label{Eq10}
a_{k}^{d}(\beta)& \leq & \int_{0}^{\infty}\frac{g_{\beta}(w)}{w^{\beta k}}dw  +\int_{1}^{\infty}\frac{g_{\beta}(w)}{w^{\beta(k-1)}}dw\nonumber\\
                & \leq & \int_{0}^{\infty}\frac{g_{\beta}(w)}{w^{\beta k}}dw  +\int_{0}^{\infty}\frac{g_{\beta}(w)}{w^{\beta(k-1)}}dw\nonumber\\
                 &= & \frac{\Gamma(1+k)}{\Gamma(1+\beta k)} + \frac{\Gamma(k)}{\Gamma(1+\beta(k-1))}\nonumber\\
\end{eqnarray}
But
$$
\frac{1}{\Gamma(1+\beta( k-1))} = \frac{1+\beta(k-1)}{\Gamma(1+\beta(k-1)+ 1)}
$$
Using D. kershaw inequality $ 1/\Gamma(x+1)<1/(x+1/2)^{1-\lambda}\Gamma(x+\lambda)$ for $x=\beta(k-1)+ 1$ and $\lambda=\beta,$ we obtain that
$$
\frac{1}{\Gamma(1 + (\beta(k-1)+ 1))}\leq \frac{1}{(\frac{3}{2}+\beta(k-1))^{1-\beta}\Gamma(1+\beta k)}.
$$
Since $ \frac{1}{(\frac{3}{2}+\beta(k-1))^{1-\beta}}\leq 1 $ and  $1+\beta(k-1)\leq 2k ,$ it follows that
\begin{eqnarray*}
\frac{1}{\Gamma(1+\beta( k-1))}& \leq & \frac{2k}{\Gamma(1+\beta k)}.\nonumber\\
\end{eqnarray*}
Multiplying both sides of the last expression by $\Gamma(k),$ we get
\begin{eqnarray*}
 \frac{\Gamma(k)}{\Gamma(1+\beta(k-1))}& \leq & 2\frac{k\Gamma(k)}{\Gamma(1+\beta k)}\nonumber\\
                                       & = & 2\frac{\Gamma(1+k)}{\Gamma(1+\beta k)}.\nonumber
 \end{eqnarray*}
 Adding $\frac{\Gamma(1+k)}{\Gamma(1+\beta k)}$ to both side of the last expression, and combine with inequality(\ref{Eq10}) give the proof of inequality (\ref{mylemmaEq1}).\\
\end{proof}
\begin{lem}\label{mylemma1'}
For $\beta \in (0, \ 1)$,   $k\in\N\cup\{0\}$ and $ d\in\{1, 2,3\}$   define
\[
a_{k}^{d}(\beta)=\int_{0}^{\infty}w^{-\beta(k-\frac{d}{4})}g_{\beta}(w)dw=\mathbb{E}(D_{1}^{-\beta(k-\frac{d}{4})}).
\]
 Then
\begin{equation}\label{mylemma1Eq2}
\frac{ a^{d}_{k}(\beta)\sqrt{ \Gamma(1+\beta(2k-\frac{d}{2}))} }{ k! }\ \leq \ \frac{ 3 \sqrt{2\beta(k-\frac{d}{4})}\sqrt{\Gamma(2\beta(1-\frac{d}{4})) } }{\Gamma(1+\beta)}2^{\beta(k-1)}\ \quad \text{for} \quad k\geq 1.
\end{equation}
\end{lem}
\begin{proof} Using the Duplication formula $ \Gamma(2x)=2^{2x-1}\Gamma(x)\Gamma(x+\frac{1}{2})/\sqrt{\pi}$, we obtain that
\begin{equation}\label{Eq11}
\Gamma\left(1+\beta\left(2k-\frac{d}{2}\right)\right)=\beta(2k-\frac{d}{2})\left[\frac{2^{2\beta(k-\frac{d}{4})-1}\Gamma\left(\beta(k-\frac{d}{4})\right)\Gamma(\beta(k-\frac{d}{4})+\frac{1}{2})}{\sqrt{\pi}}\right]
\end{equation}
This combining with (\ref{mylemmaEq1}) yields that
\begin{equation}\label{Eq15}
\frac{a_{k}^{d}(\beta)\sqrt{\Gamma(1+\beta(2k- \frac{d}{2}))}}{k!}\leq 3\sqrt{\beta(2k-\frac{d}{2})}\left[\frac{2^{\beta(k-\frac{d}{4})-\frac{1}{2}}\sqrt{\Gamma(\beta(k-\frac{d}{4}))}\sqrt{\Gamma(\beta(k-\frac{d}{4})+\frac{1}{2})}}{(\pi)^{\frac{1}{4}}\Gamma(1+\beta k)}\right]
\end{equation}
 Using the relationship between the Beta and Gamma functions $\Gamma(x)\Gamma(y)=B(x,y)\Gamma(x,y)$ with $x=\beta k +\frac{1}{2}-\frac{\beta d}{4}$ and $y=\frac{1}{2}+\frac{\beta d}{4},$ we obtain that
 \begin{eqnarray*}
 \frac{\Gamma(\beta(k-\frac{d}{4})+\frac{1}{2}) \Gamma(\frac{1}{2}+\frac{\beta d}{4})}{\Gamma(1+\beta k)}
 & =& B\left( \beta(k-\frac{d}{4})+\frac{1}{2} ,\frac{1}{2}+\frac{\beta d}{4}\right)  \nonumber\\
 &= & \int_{0}^{1}t^{\beta k +\frac{1}{2}-\frac{\beta d}{4}-1}(1-t)^{\frac{1}{2}+\frac{\beta d}{4}}dt   \nonumber\\
 &= & \int_{0}^{1}t^{\beta k}t^{\frac{1}{2}-\frac{\beta d}{4}-1}(1-t)^{\frac{1}{2}+\frac{\beta d}{4}}dt   \nonumber\\
 &\leq & \int_{0}^{1}t^{\beta}t^{\frac{1}{2}-\frac{\beta d}{4}-1}(1-t)^{\frac{1}{2}+\frac{\beta d}{4}}dt   \nonumber\\
 &= & \int_{0}^{1}t^{\beta+\frac{1}{2}-\frac{\beta d}{4}-1}(1-t)^{\frac{1}{2}+\frac{\beta d}{4}}dt   \nonumber\\
 &=& B\left(\beta+\frac{1}{2}-\frac{\beta d}{4},\frac{1}{2}+\frac{\beta d}{4}\right) \nonumber\\
 &=& \frac{\Gamma(\beta+\frac{1}{2}-\frac{\beta d}{4})\Gamma(\frac{1}{2}+\frac{\beta d}{4})}{\Gamma(1+\beta)}\nonumber\\
 \end{eqnarray*}
 It follows that
 \begin{equation}\label{Eq12}
 \frac{\Gamma(\beta(k-\frac{d}{4})+\frac{1}{2})}{\Gamma(1+\beta k)}\leq \frac{\Gamma(\beta(1-\frac{d}{4})+\frac{1}{2})}{\Gamma(1+\beta)}
 \end{equation}
 On the other hand, by repeating again the previous arguments with $x=\beta(k -\frac{d}{4})$ and $y=1+\frac{\beta d}{4}$ we obtain that
 \begin{equation}\label{Eq13}
  \frac{\Gamma(\beta(k-\frac{d}{4}))}{\Gamma(1+\beta k)}\leq \frac{\Gamma(\beta(1-\frac{d}{4}))}{\Gamma(1+\beta)}.
 \end{equation}
Inequalities (\ref{Eq12}) and  (\ref{Eq13}) combined give
\begin{equation}\label{Eq15}
\frac{\Gamma(\beta(k-\frac{d}{4}))\Gamma(\beta(k-\frac{d}{4})+\frac{1}{2})}{\left[\Gamma(1+\beta k)\right]^2}\leq \frac{\Gamma(\beta(1-\frac{d}{4})+\frac{1}{2})\Gamma(\beta(1-\frac{d}{4}))}{\left[\Gamma(1+\beta)\right]^2}.
\end{equation}
  Duplication formula with $x=\beta(1-\frac{d}{4})$ give
  \begin{equation}\label{Eq16}
  \Gamma(\beta(1-\frac{d}{4})+\frac{1}{2})\Gamma(\beta(1-\frac{d}{4}))=2^{1-2\beta(1-\frac{d}{4})}\sqrt{\pi}\Gamma(2\beta(1-\frac{d}{4})).
  \end{equation}
 Combining (\ref{Eq15}),(\ref{Eq16}) and  we obtain that
 $$
 \frac{\Gamma(\beta(k-\frac{d}{4}))\Gamma(\beta(k-\frac{d}{4})+\frac{1}{2})}{\left[\Gamma(1+\beta k)\right]^2}\leq
 \frac{2^{1-2\beta(1-\frac{d}{4})}\sqrt{\pi}\Gamma(2\beta(1-\frac{d}{4}))}{\left[\Gamma(1+\beta)\right]^2}.
 $$
 Taking square root of both side of the last expression, we get
\begin{equation}\label{Eq14}
\frac{2^{\beta(1-\frac{d}{4})-\frac{1}{2}}\sqrt{\Gamma(\beta(k-\frac{d}{4}))}\sqrt{\Gamma(\beta(k-\frac{d}{4})+\frac{1}{2})}}{(\pi)^{\frac{1}{4}}\Gamma(1+\beta k)}\leq \frac{\sqrt{\Gamma(2\beta(1-\frac{d}{4}))}}{\Gamma(1+\beta)}.
\end{equation}
Inequalities (\ref{Eq15}) and (\ref{Eq14}) complete the proof of (\ref{mylemma1Eq2}).
\end{proof}

The next lemma will also be needed in the proof of our main theorem in the next section.

\begin{lem}\label{mylemma2}
For every $\beta\in(0, \ 1)$ and $n, k\in\N\cup\{0\},$ and $d\in\{1,2,3\}$ satisfying the assumption of Proposition \ref{propostionforInt}, define
\[
b_{k,n}(\beta)=\int_{0}^{t}e^{-\gamma s}s^{\beta(k+n -\frac{d}{2})}ds.
\]
\end{lem}
Then
\begin{equation}\label{mylemma2Eq1}
b_{k,n}(\beta)\  \leq\ \left[b_{k,k}(\beta) \right]^{\frac{1}{2}}\left[b_{n,n}(\beta) \right]^{\frac{1}{2}}
\end{equation}
and
\begin{equation}\label{mylemma2Eq2}
b_{k,k}(\beta)\ \leq \ \left(\frac{1}{\gamma}\right)^{1+2\beta(k-\frac{d}{4})}\Gamma(1+\beta(2k-\frac{d}{2})).
\end{equation}
\begin{proof}\textbf{Proof of inequality (\ref{mylemma2Eq1}) of Lemma\ref{mylemma2}.}
\begin{eqnarray*}
b_{k,n}(\beta)& =&\int_{0}^{t}e^{-\gamma s}s^{\beta(k+n-\frac{d}{2})}ds\\
              & =& \int_{0}^{t}\left(e^{-\frac{\gamma s}{2}}s^{\beta(k-\frac{d}{4})} \right)\left(e^{-\frac{\gamma s}{2}}s^{\beta(n-\frac{d}{4})} \right)ds\\
              & \leq & \left(\int_{0}^{t}e^{-\gamma s}s^{2\beta(k-\frac{d}{4})}ds \right)^{\frac{1}{2}}\left(\int_{0}^{t}e^{-\gamma s}s^{2\beta(n-\frac{d}{4})}ds \right)^{\frac{1}{2}}\\
              &=&\left[b_{k,k}(\beta)\right]^{\frac{1}{2}}\left[b_{n,n}(\beta)\right]^{\frac{1}{2}}.\\
\end{eqnarray*}
\textbf{Proof of inequality (\ref{mylemma2Eq2}) of Lemma\ref{mylemma2}.}
\begin{eqnarray*}
b_{k,k}(\beta)& = &\int_{0}^{t}e^{-\gamma s}s^{2\beta k -\frac{\beta d}{2}}ds\\
              & = &\int_{0}^{\gamma t}e^{-w}\left(\frac{w}{\gamma}\right)^{2\beta k -\frac{\beta d}{2}}\frac{dw}{\gamma}\\
              & = &\left(\frac{1}{\gamma}\right)^{2(\beta k -\frac{d}{4})+1}\int_{0}^{\gamma t}e^{-w}w^{\beta(2k-\frac{d}{2})}dw\\
              & \leq & \left(\frac{1}{\gamma}\right)^{2(\beta k -\frac{d}{4})+1}\Gamma(\beta(2k-\frac{d}{2})+1).\\
\end{eqnarray*}
\end{proof}

\section{Intermittency fronts}\label{intermitteny-fronts}

 Here we state and prove our main result on the intermittency fronts for the solution of equation \eqref{tfspde}. Our results generalize the work of Jebessa B. Mejina and Erkan Nane see theorem4.1 in \cite{JebesaAndNane1}. In \cite{JebesaAndNane1} the authors proved the result for $d=1$ and $\alpha=2.$ With the aid of Lemma\ref{mylemma1}, Lemma\ref{mylemma1'} and Lemma\ref{mylemma2} we are able to overcome the difficulties in their methods and extend the  result for $d\in\{1,2,3\}$ and $\alpha=2.$
  Assume that $\sigma(\cdot)$ in \eqref{tfspde} satisfies the following global Lipschitz condition, i.e. there exists a generic positive constant $\mbox{Lip}_\sigma$ such that:
\begin{equation}\label{Eq:Cond_a}
|\sigma(x)-\sigma(y)|\le \mbox{Lip}_\sigma\|x-y\|\quad\mbox{for all }\,\,x,\,y\in\R^d.
\end{equation}
Clearly, (\ref{Eq:Cond_a}) implies the uniform linear growth condition of $\sigma(\cdot).$ Recall the definition of $\mathscr{L}(\theta)$ from \eqref{intermittecy-front-L}.

We first state a proposition that implies that the solution of equation \eqref{tfspde} is square integrable over time in the language of partial differential equations.
\begin{prop}[Proposition 4.2 in \cite{JebesaAndNane1}]Assume that $\alpha\in (0,2]$,   and $d<\min\{2,\beta^{-1}\}\a$, then $u_t\in L^2(\R)$ a.s. for all $t\geq 0;$ in
fact, for any fixed $\epsilon\in (0,1)$ and $t\geq 0,$
\begin{equation}\label{solutionL2integrability}
\E\left(||u_t||^2_{L^2(\R^{d})}\right)\leq \epsilon^{-1}||u_0||^2_{L^2(\R^{d})}\exp\left(\bigg[\frac{C^\ast\Gamma(1-\beta d/\a)\mbox{Lip}_\sigma^2}{1-\epsilon}\bigg]^{\frac{1}{1-\beta d/\a}} t\right)
\end{equation}
\end{prop}

The proof of Theorem \ref{maintheoremtwo} requires the following ``weighted stochastic Young's inequality'' which is an extension of Proposition  8.3 in \cite{khoshnevisan-cbms}.

\begin{prop}\label{propostionforInt}
Let $\alpha=2 $ and $ d< \min\{2, \beta^{-1}\}\alpha$.
Define for all $\gamma > 0, c\in\R^{d},$ and $\Phi\in\mathcal{L}^{\beta,2},$
\begin{equation*}
\mathcal{N}_{\gamma,c}(\Phi):= \sup_{t\geq 0}\sup_{x\in{\R^{d}}}\bigg[e^{-\gamma t + c.x}\E\left(|\Phi_t(x)|^2\right)\bigg]^{1/2}.
\end{equation*}
Then,
\begin{equation*}
\mathcal{N}_{\gamma,c}(G\circledast \Phi)\leq C_{d}(c,\gamma,\beta)\mathcal{N}_{\gamma,c}(\Phi)\ \ \ \mbox{for all}\ \left(\frac{\gamma}{2}\right)^\beta >\frac{\nu \|c\|^2}{2},
\end{equation*}
where $C_{d}(c,\gamma,\beta)$ is a finite constant that depends on $d, \|c\|,\gamma,$ and $\beta$.
\end{prop}

Using  Lemma\ref{mylemma1},Lemma\ref{mylemma1'}, Lemma\ref{mylemma2} and the last two propositions, we are now ready to give the proof of  \ref{maintheoremtwo}.
\begin{proof}
Using $p_{u}(y)=\frac{e^{-\frac{\|y\|^{2}}{4u\nu}}}{(4\pi u\nu)^{\frac{d}{2}}},$  direct computations yield
\[
\int_{\R^{d}}e^{-c.y}\left[ p_{u}(y) \right]^2dy= \frac{e^{\frac{u\nu\|c\|^{2}}{2}}}{(8\pi u\nu)^{\frac{d}{2}}}.
\]
Observe that
\begin{eqnarray}
[G_s(y)]^2 &=& \int_0^\infty p_u(y)f_{E_s}(u)du\int_0^\infty p_v(y)f_{E_s}(v)dv\\ \nonumber
&=&\int_0^\infty\int_0^\infty p_u(y)p_v(y)f_{E_s}(u)f_{E_s}(v)dudv.
\end{eqnarray}
We use Holder's inequality to obtain that
\begin{eqnarray}\label{pro1}
 &      & \int_{\R^{d}}e^{-c.y}\left[ G_{s}(y) \right]^{2}dy \nonumber\\
 & = & \int_{\R^{d}}\int_0^\infty\int_0^\infty\left(e^{-\frac{c.y}{2}}  p_u(y)f_{E_s}(u)\right)\left(e^{-\frac{c.y}{2}}p_v(y)f_{E_s}(v)\right)dudvdy\nonumber\\
 &  = &   \int_0^\infty\int_0^\infty\int_{\R^{d}}\left[\left(e^{-\frac{c.y}{2}}  p_u(y)\right)\left(e^{-\frac{c.y}{2}}p_v(y)\right)dy\right]f_{E_s}(u)f_{E_s}(v)dudv\nonumber\\
 & \leq &  \int_0^\infty\int_0^\infty\left(\int_{\R^{d}}e^{-c.y}\left[p_u(y)\right]^{2}dy\right)^{\frac{1}{2}}\left(\int_{\R^{d}}e^{-c.y}\left[p_v(y)\right]^{2}dy\right)^{\frac{1}{2}}f_{E_s}(u)f_{E_s}(v)dudv\nonumber\\
 & = & \left[\int_0^\infty\left(\int_{\R^{d}}e^{-c.y}\left[p_u(y)\right]^{2}dy\right)^{\frac{1}{2}}f_{E_s}(u)du\right]^{2}\nonumber\\
&  =  & \left[ \int_{0}^\infty \left(  \frac{e^{\frac{u\nu\|c\|^{2}}{2}}}{(8\pi u\nu)^{\frac{d}{2}}} \right)^{\frac{1}{2}}f_{E_{s}}(u)du  \right]^{2} \nonumber\\
& =  &\frac{1}{(8\pi\nu)^{\frac{d}{2}}}  \left[ \int_{0}^\infty\frac{e^{\frac{u\nu\|c\|^{2}}{4}}}{u^{\frac{d}{4}}}f_{E_{s}}(u)du  \right]^{2} \nonumber\\
& \leq &  \frac{1}{(8\pi\nu)^{\frac{d}{2}}}\left[ \sum_{k=0}^{\infty}\frac{\left( \frac{\nu\|c\|^{2}}{4} \right)^{k}}{k!}\int_{0}^{\infty}u^{k-\frac{d}{4}}f_{E_{s}}(u)du\right]^{2} \nonumber\\
& = &  \frac{1}{(8\pi\nu)^{\frac{d}{2}}}\left[ \sum_{k=0}^{\infty}\frac{\left( \frac{\nu\|c\|^{2}}{4} \right)^{k}}{k!}s^{\beta(k-\frac{d}{4})}\int_{0}^{\infty}w^{-\beta(k-\frac{d}{4})}g_{\beta}(w)dw\right]^{2} \nonumber\\
& = &  \frac{1}{(8\pi\nu)^{\frac{d}{2}}}\left[ \sum_{k=0}^{\infty}\frac{\left( \frac{\nu\|c\|^{2}}{4} \right)^{k}}{k!}s^{\beta(k-\frac{d}{4})}a^{d}_{k}(\beta)\right]^{2} \nonumber\\
 & \leq & \frac{s^{-\frac{\beta d}{2}}}{(8\pi\nu)^{\frac{d}{2}}} \sum_{k,n=0}^{\infty}\frac{a^{d}_{k}(\beta)a_{n}^{d}(\beta)}{n!k!}\left( \frac{\nu\|c\|^{2}}{4}\right)^{k+n}s^{\beta(k+n)}\nonumber\\
\end{eqnarray}
where
$$a^{d}_{k}(\beta)= \int_{0}^{\infty}w^{-\beta(k-\frac{d}{4})}g_{\beta}(w)dw    $$
 From the inequality (\ref{pro1}), we obtain that
\begin{eqnarray}\label{pro2}
&         & \int_{0}^{t}e^{-\gamma s}ds\int_{\R^{d}}e^{-c.y}\left[ G_{s}(y) \right]^{2}dy \nonumber\\
& \leq &  \frac{1}{(8\pi\nu)^{\frac{d}{2}}} \sum_{k,n=0}^{\infty}\frac{a^{d}_{k}(\beta)a_{n}^{d}(\beta)}{n!k!}\left( \frac{\nu\|c\|^{2}}{4}\right)^{k+n}\int_{0}^{t}s^{\beta(k+n-\frac{d}{2})}e^{-\gamma s}ds\nonumber \\
& = & \frac{1}{(8\pi\nu)^{\frac{d}{2}}} \sum_{k,n=0}^{\infty}\frac{a^{d}_{k}(\beta)a_{n}^{d}(\beta)}{n!k!}\left( \frac{\nu\|c\|^{2}}{4}\right)^{k+n}b_{n,k}(\beta)\nonumber \\
\end{eqnarray}
Where $$b_{n,k}(\beta)=\int_{0}^{t}s^{\beta(k+n-\frac{d}{2})}e^{-\gamma s}ds.$$

 Combining inequality (\ref{pro2}) and inequality (\ref{mylemma2Eq1}) of Lemma\ref{mylemma2} we obtain that
\begin{eqnarray*}
&        & \int_{0}^{t}e^{-\gamma s}\left[ \int_{ \R^{d} } e^{-c.y}\left[ G_{s}(y) \right]^{2}dy\right]ds   \nonumber \\
& \leq &  \frac{1}{(8\pi\nu)^{\frac{d}{2}}} \sum_{k,n=0}^{\infty}\frac{a^{d}_{k}(\beta)a_{n}^{d}(\beta)}{n!k!}\left( \frac{\nu\|c\|^{2}}{4}\right)^{k+n}\left[b_{k,k}(\beta)\right]^{\frac{1}{2}}\left[b_{n,n}(\beta)\right]^{\frac{1}{2}}\nonumber\\
& = &  \frac{1}{(8\pi\nu)^{\frac{d}{2}}} \left[\sum_{k=0}^{\infty}\frac{a^{d}_{k}(\beta)}{k!}\left( \frac{\nu\|c\|^{2}}{4}\right)^{k}\left[b_{k,k}(\beta)\right]^{\frac{1}{2}}\right]^{2}\nonumber\\
\end{eqnarray*}
Using inequality(\ref{mylemma2Eq2}) of Lemma\ref{mylemma2}, the last inequality can be improved to
\begin{eqnarray}\label{pro3}
&        & \int_{0}^{t}e^{-\gamma s}\left[ \int_{ \R^{d} } e^{-c.y}\left[ G_{s}(y) \right]^{2}dy\right]ds   \nonumber \\
& \leq &  \frac{1}{(8\pi\nu)^{\frac{d}{2}}} \left[\sum_{k=0}^{\infty}\frac{a^{d}_{k}(\beta)\sqrt{\Gamma(1+\beta(2k-\frac{d}{2}))}}{k!}\left( \frac{\nu\|c\|^{2}}{4}\right)^{k} \left(\frac{1}{\gamma}\right)^{\frac{1}{2}+\beta(k-\frac{d}{4})} \right]^{2}\nonumber\\
\end{eqnarray}
Next, using inequality(\ref{mylemma1Eq2}) of Lemma\ref{mylemma1'} and the fact that $\sqrt{2k-\frac{d}{2}} \leq 2^{k}$ for every $k\geq1$, inequality(\ref{pro3}) becomes
\begin{eqnarray}\label{pro4}
&        & \int_{0}^{t}e^{-\gamma s}\left[ \int_{ \R^{d} } e^{-c.y}\left[ G_{s}(y) \right]^{2}dy\right]ds   \nonumber \\
    & \leq &  \frac{\gamma^{\frac{\beta d}{2}-1}}{(8\pi\nu)^{\frac{d}{2}}} \left[a_{0}^{d}(\beta)\sqrt{\Gamma(1-\frac{\beta d}{2})}+ \frac{3\sqrt{\beta\Gamma(2\beta(1-\frac{d}{4}))}}{2^{\beta}\Gamma(1+\beta)} \sum_{k=1}^{\infty}\left(  \frac{2^{\beta}\nu\|c\|^{2} }{4\gamma^{\beta}} \right)^{k}\sqrt{2k-\frac{d}{2}} \right]^{2}\nonumber\\
 & \leq &  \frac{\gamma^{\frac{\beta d}{2}-1}}{(8\pi\nu)^{\frac{d}{2}}} \left[a_{0}^{d}(\beta)\sqrt{\Gamma(1-\frac{\beta d}{2})}+ \frac{3\sqrt{\beta\Gamma(2\beta(1-\frac{d}{4}))}}{2^{\beta}\Gamma(1+\beta)} \sum_{k=1}^{\infty}\left(  \frac{2^{\beta+1}\nu\|c\|^{2} }{4\gamma^{\beta}} \right)^{k}\right]^{2}\nonumber\\
 & \leq &  \frac{M^{2}\gamma^{\frac{\beta d}{2}-1}}{(8\pi\nu)^{\frac{d}{2}}} \left[1+ \sum_{k=1}^{\infty}\left(  \frac{2^{\beta-1}\nu\|c\|^{2} }{\gamma^{\beta}} \right)^{k} \right]^{2}\nonumber\\
 & = & \left[  \frac{M\gamma^{\frac{\beta d}{4}-\frac{1}{2}}}{(8\pi\nu)^{\frac{d}{4}}} \sum_{k=0}^{\infty}\left(  \frac{2^{\beta-1}\nu\|c\|^{2} }{\gamma^{\beta}} \right)^{k} \right]^{2}\nonumber\\
\end{eqnarray}
where
$$
M=\max\left\{a_{0}^{d}(\beta)\sqrt{\Gamma(1-\frac{\beta d}{2})},  \frac{3\sqrt{\beta\Gamma(2\beta(1-\frac{d}{4}))}}{2^{\beta}\Gamma(1+\beta)}\right\}
$$
The last series converges if and only if $ \left(\frac{\gamma}{2}\right)^{\beta}> \frac{\nu\|c\|^{2}}{2}. $
Therefore from inequality(\ref{pro4}), we obtain that
\begin{eqnarray}\label{pro5}
&&e^{-\gamma t + c.x}\E(|(G\circledast \Phi)_t(x)|^2)\nonumber\\
&=& e^{-\gamma t + c.x}\int_{0}^{t}\left[ \int_{\R^{d}}\left[G_{t-s}(y-x) \right]^{2}\E(| \Phi_{s}(y)|^{2})dy  \right]ds \nonumber\\
&\leq & \left[ \mathcal{N}_{\gamma, c}(\Phi) \right]^{2}\int_{0}^{t}e^{-\gamma s}\int_{\R^{d}}e^{-c.y}\left[ G_{s}(y)\right]^{2}dyds \nonumber\\
&\leq&  \left[ \mathcal{N}_{\gamma, c}(\Phi) \right]^{2}\left[  \frac{M\gamma^{\frac{\beta d}{4}-\frac{1}{2}}}{(8\pi\nu)^{\frac{d}{4}}} \sum_{k=0}^{\infty}\left(  \frac{2^{\beta-1}\nu\|c\|^{2} }{\gamma^{\beta}} \right)^{k} \right]^{2}\nonumber\\
\end{eqnarray}
The right-hand side is independent of $(x,t).$ Therefore, by optimizing over $(x,t)$ and then square roots of both side , we get
\begin{equation*}
\mathcal{N}_{\gamma,c}(G\circledast \Phi)\leq C_{d}(c,\gamma,\beta)\mathcal{N}_{\gamma,c}(\Phi)
\end{equation*}
 with
\begin{equation}\label{pro6}
C_{d}(c,\gamma,\beta)= \frac{M\gamma^{\frac{\beta d}{4}-\frac{1}{2}}}{(8\pi\nu)^{\frac{d}{4}}} \sum_{k=0}^{\infty}\left(  \frac{2^{\beta-1}\nu\|c\|^{2} }{\gamma^{\beta}} \right)^{k}
                                      =\frac{M\gamma^{\frac{\beta d}{4}-\frac{1}{2}}}{(8\pi\nu)^{\frac{d}{4}}\left(1- \frac{2^{\beta-1}\nu\|c\|^{2} }{\gamma^{\beta}}\right)}.
\end{equation} Which complete the proof of Proposition \ref{propostionforInt}.
\end{proof}
The next corollary is a generalization  of Corollary 4.4 in \cite{JebesaAndNane1}.
\begin{coro}\label{coro3}
If $\|c\|^{1/\beta-d/2}>\mbox{Lip}_\sigma \sqrt{\frac{M^{2}}{(2\nu)^{\frac{1}{\beta}-\frac{d}{2}}  (1-2^{\beta-2})^{2}(8\pi\nu)^{\frac{d}{2}}}}$, then the solution to the tfspde \eqref{tfspde} for  $\a=2$ satisfies
\begin{equation}
\E(|u_t(x)|^2)\leq A(\|c\|,\beta)\exp\left(-\|c\|\|x\| + (2\nu c^2)^{1/\beta}t\right),
\end{equation}
simultaneously for all $x\in\R^{d}$ and $t\geq 0,$ where $A(\|c\|,\beta)$ is a finite constant that depends only on $\|c\|$ and $\beta.$
\end{coro}
\begin{proof}
The proof generalizes some of the ideas used in the proof of Corollary 4.4 in \cite{JebesaAndNane1}.
Recall that for all $\gamma > 0$
\begin{eqnarray}
\mathcal{N}_{\gamma,c}(u^{(n+1)})&\leq& [\mathcal{N}_{\gamma,c}(G_t\ast u_0)] + [\mathcal{N}_{\gamma,c}(G\circledast \sigma(u^{(n)})]\nonumber\\
&\leq&[\mathcal{N}_{\gamma,c}(G_t\ast u_0)] + C_{d}(\|c\|,\gamma,\beta)[\mathcal{N}_{\gamma,c}(\sigma(u^{(n)}))],\nonumber
\end{eqnarray}
using Proposition \ref{propostionforInt}. Because $\sigma(z)\leq \mbox{Lip}_\sigma|z|$ for all $z\in\R^{d},$
$$\mathcal{N}_{\gamma,c}(\sigma(u^{(n)})\leq \mbox{Lip}_\sigma\mathcal{N}_{\gamma,c}(u^{(n)}).$$
Also,
\begin{eqnarray}
e^{-\gamma t + c.x}(G_t\ast |u_0|)(x)&=&e^{-\gamma t}\int_{\R^{d}}G_t(y-x)e^{-c.(y-x)}e^{cy}|u_0(y)|dy\nonumber\\
&\leq& e^{-\gamma t}\mathcal{N}_{0,c}(u_0)\int_{\R^{d}}e^{c.z}G_t(z)dz\nonumber\\
&=& e^{-\gamma t}E_\beta(\nu \|c\|^2t^\beta)\mathcal{N}_{0,c}(u_0).
\end{eqnarray}
We take $\gamma^\beta := 2\nu\| c\|^2$ to see that for all integers $k\geq 0$
$$e^{-\gamma t}E_\beta(\nu \|c\|^2t^\beta) = \sum_{k=0}^\infty \frac{\nu^k \|c\|^{2k}t^{\beta k}e^{-\gamma t}}{\Gamma(1+\beta k)}=\sum_{k=0}^\infty \frac{\nu^k \|c\|^{2k}}{\gamma^{\beta k}}\frac{u^{\beta k}e^{-u}}{\Gamma(1+\beta k)}\leq 2,$$
since $\frac{u^{\beta k}e^{-u}}{\Gamma(1+\beta k)}<1.$  From equation \ref{pro5}
$$  C_{d}(\|c\|,\gamma,\beta)=\frac{M\gamma^{\frac{\beta d}{4}-\frac{1}{2}}}{(8\pi\nu)^{\frac{d}{4}}\left(1- \frac{2^{\beta-1}\nu\|c\|^{2} }{\gamma^{\beta}}\right)}
                                                 =\frac{M(2\nu)^{\frac{d}{4}-\frac{1}{2\beta}}\|c\|^{\frac{d}{2}-\frac{1}{\beta}}}{(1-2^{\beta-2})(8\pi\nu)^{\frac{d}{4}}}.  $$
and by our assumption $$C_{d}(\|c\|,\gamma,\beta)\mbox{Lip}_\sigma<1.$$
For every $ \theta \in [0, 2\pi ]$ and $\varphi\in [0, \pi ]$,  we define  $ c(\theta, \varphi)$
\footnote{ For $d=2$ we define $ c(\theta):=\|c\|(\cos(\theta), \sin(\theta))$  for every $\theta\in [0, 2\pi]$ while for the case $d=1$ we consider $c(\theta):=|c|\cos(\theta) $ where $\theta\in \{0,\pi\} $ }
by the expression
$$ c( \theta, \varphi):=\|c\| (\cos(\varphi) \cos(\theta),\cos(\varphi) \sin(\theta), \sin(\varphi)).$$
 Since for every $\theta\in[0, 2\pi]$ and $\varphi\in[0, \pi],$ we have
$ \|c(\theta,\varphi)\|= \|c\|,$ then   for  all integers $n\geq 0$  we have
\begin{equation}\label{Eq20}
\mathcal{N}_{(2\nu \|c\|^2)^{1/\beta},c(\theta,\varphi)}(u^{(n+1)})\leq 2\mathcal{N}_{0,c(\theta,\varphi)}(u_{0}) + C_{d}(\|c\|,\gamma,\beta)\mbox{Lip}_\sigma\mathcal{N}_{(2\nu \|c\|^2)^{1/\beta},c(\theta,\varphi)}(u^{(n)}).\end{equation}
Since $u_0$ has compact support, there is some constant $R>0$ such that $ u_{0}(x)=0$ whenever $\|x\|\geq R.$ Hence we obtain that
\begin{eqnarray*}
\sup_{\theta\in[0, 2\pi]}\sup_{\varphi\in[0, \pi]}\mathcal{N}_{0,c(\theta,\varphi)}(u_0):
&=&\sup_{\theta\in[0, 2\pi]}\sup_{\varphi\in[0, \pi]}\sup_{x\in\R^{d}} e^{c(\theta,\varphi).x}u_0(x)\nonumber\\
&=& \sup_{\theta\in[0, 2\pi]}\sup_{\varphi\in[0, \pi]}\sup_{x\in supp(u_{0})} e^{c(\theta,\varphi).x}u_0(x)\nonumber\\
&\leq & \sup_{x\in supp(u_{0})} e^{\|c(\theta,\varphi)\|\|x\|}u_0(x)\nonumber\\
&=& \sup_{x\in supp(u_{0})} e^{\|c\|\|x\|}u_0(x)\nonumber\\
&\leq & e^{R\|c\|}\|u_{0}\|_{\infty}.\nonumber\\
\end{eqnarray*}

 Since $e^{R\|c\|}\|u_{0}\|_{\infty}<\infty,$  it follows from inequality(\ref{Eq20}) that
$$\sup_{n\geq 0}\sup_{\theta\in[0, 2\pi]}\sup_{\varphi\in[0, \pi]}\mathcal{N}_{(2\nu \|c\|^2)^{1/\beta},c(\theta,\varphi)}(u^{(n+1)})<\infty.$$
Since $u_t^{(n+1)}(x)$ converges to $u_t(x)$ in $L^2(\Omega)$ as $n\rightarrow \infty$  Fatou's lemma implies that
$$\sup_{\theta\in[0, 2\pi]}\sup_{\varphi\in[0, \pi]}\mathcal{N}_{(2\nu \|c\|^2)^{1/\beta},c(\theta, \varphi)}(u)<\infty.$$
Since every $x\in\R^{d}$ can be written as $x=\|x\|(\cos(\phi_{x})\cos(\theta_{x}),\cos(\phi_{x})\sin(\theta_{x}),\sin(\phi_{x}))$ and the preceding supremum is independent of $\theta$ and $\phi$, then in particular for $\theta=\theta_{x} $ and $\phi=\phi_{x}$ we obtain that $c(\theta_{x},\phi_{x}).x=\|c\|\|x\|.$ The corollary follows readily from this fact.
\end{proof}
We are ready to prove Theorem \ref{maintheoremtwo}. We do this  in two steps adapting the method in \cite[Chapter 8]{khoshnevisan-cbms} with crucial nontrivial changes: First we derive \eqref{intermittencypart1}; and then we establish \eqref{intermittencypart2}.
\begin{proof}[{\bf Proof of  \eqref{intermittencypart1}}]
 Since $u_0$ has compact support, it follows that $|u_0(x)|= O(e^{\|c\|\|x\|})$ for all $\|c\|>0.$ Therefore, we may apply Corollary \ref{coro3} to an arbitrary
$\|c\|^{1/\beta-d/2}>\mbox{Lip}_\sigma \sqrt{\frac{M^{2}}{(2\nu)^{\frac{1}{\beta}-\frac{d}{2}}  (1-2^{\beta-2})^{2}(8\pi\nu)^{\frac{d}{2}}}}:= \mbox{Lip}_\sigma c_0$ in order to see that
\begin{eqnarray}
\mathscr{L}(\theta)&=& \limsup_{t\rightarrow\infty}\frac{1}{t}\sup_{\|x\|>\theta t}\log \E\left(|u_t(x)|^2\right)\leq - \sup_{\|c\|>(\mbox{Lip}_\sigma c_0)^{2\beta/(2-\beta d)}}\bigg[\theta\| c\| - (2\nu\|c\|^2)^{1/\beta}\bigg]\nonumber\\
&\leq&-\bigg[\theta (\mbox{Lip}_\sigma c_0)^{2\beta/(2-\beta d)}) - (2\nu)^{1/\beta}(\mbox{Lip}_\sigma c_0)^{4/(2-\beta d)})\bigg],
\end{eqnarray}
obtained by setting $\|c\|:= (\mbox{Lip}_\sigma c_0)^{2\beta/(2-\beta d)}$ in the maximization problem of the first line of preceding display. The right-most quantity is strictly negative when
$$\theta >  (2\nu)^{1/\beta}(\mbox{Lip}_\sigma c_0)^{2(\frac{2-\beta}{2-\beta d})};$$
this proves\eqref{intermittencypart1}.
\end{proof}
\begin{proof}[{\bf Proof of  \eqref{intermittencypart2}.}] We have that
\begin{eqnarray}\label{intermittencyproof}
&&\E(|u_t(x)|^2)\nonumber\\
&\geq&|(G_t\ast u_0)(x)|^2 + L_\sigma^2\int_0^tds\int_{\R^{d}}dy[G_{t-s}(y-x)]^2\E(|u_s(y)|^2).
\end{eqnarray}
For all $t>0$ and $x\in\R^{d}.$ Define $\mathbb{K}^{+}_{\theta t}: =\mathbb{R}^{d-1}\times[\theta t, \infty)$,  $\mathbb{K}^{-}_{\theta t}: =\mathbb{R}^{d-1}\times(-\infty, -\theta t]$ and
$\mathbb{K}_{\theta t}=\mathbb{K}^{+}_{\theta t}\cup \mathbb{K}^{-}_{\theta t} $  for every $t>0$ and $\theta>0.$ Then  if $x,y\in\R^{d}, 0\leq s\leq t,$ and $\theta\geq 0,$ we have
$$1_{\mathbb{K}^{+}_{\theta t}}(x)\geq 1_{\mathbb{K}^{+}_{\theta(t-s)}}(x-y)\cdot 1_{\mathbb{K}^{+}_{\theta s}}(y).$$
This is a consequence of the triangle inequality. Therefore,
\begin{eqnarray}
&&\int_{\mathbb{K}^{+}_{\theta t}}\int_0^tds\int_{\R^{d}}dy[G_{t-s}(y-x)]^2\E(|u_s(y)|^2)\nonumber\\
& = & \int_0^tds\int_{\R^{d}\times \R^{d}}dy[G_{t-s}(y-x)]^2\E(|u_s(y)|^2)1_{\mathbb{K}^{+}_{\theta t}}(x)dydx \nonumber\\
&\geq& \int_0^tds\left(\int_{\mathbb{K}^{+}_{\theta(t-s)}}[G_{t-s}(y)]^2 dy\right)\left( \int_{\mathbb{K}^{+}_{\theta s}}\E(|u_s(y)|^2)dy\right)\nonumber\\
\end{eqnarray}
This and \eqref{intermittencyproof} together show that the function
\begin{equation}
M_+(t) := \int_{\mathbb{K}^{+}_{\theta t}} \E(|u_t(y)|^2)dy
\end{equation}
satisfies the following renewal inequality:
\begin{equation}\label{positivepart}
M_+(t)\geq \int_{\mathbb{K}^{+}_{\theta t}}|(G_t\ast u_0)(x)|^2dx + L_\sigma^2(T\ast M_+)(t),
\end{equation}
with $$T(t) := \int_{\mathbb{K}^{+}_{\theta t} }[G_t(z)]^2dz.$$
Because of symmetry we can write $T(t) = \int_{\mathbb{K}^{-}_{\theta t}}[G_t(z)]^2dz.$ Therefore, a similar argument shows that the function
$$M_{-}(t) := \int_{\mathbb{K}^{-}_{\theta t}}\E(|u_s(y)|^2)dy,$$
satisfies the following renewal inequality:
\begin{equation}\label{negativepart}
M_{-}(t)\geq \int_{\mathbb{K}^{-}_{\theta t}}|(G_t\ast u_0)(x)|^2dx + L_\sigma^2(T\ast M_{-})(t).
\end{equation}
Define
$$M(t) := \int_{\mathbb{K}_{\theta t}} \E(|u_t(y)|^2)dy = M_+(t) + M_{-}(t),$$
Define $\mathcal{L}\phi$ to be the Laplace transform of any measurable function $\phi:\R_+\rightarrow\R_+.$ That is,
$$(\mathcal{L}\phi)(\lambda) = \int_0^\infty e^{-\lambda t}\phi(t)dt\ \ \ (\lambda \geq 0).$$
Then, we have the following inequality of Laplace transforms: For every $\lambda\geq 0,$
\begin{eqnarray}\label{Eq200}
(\mathcal{L}M)(\lambda)& = & (\mathcal{L}M_{+})(\lambda)+  (\mathcal{L}M_{-})(\lambda)\nonumber  \\
&\geq& \int_0^\infty e^{-\lambda t}dt\int_{\mathbb{K}_{\theta t}} dx |(G_t\ast u_0)(x)|^2 + L_\sigma^2 (\mathcal{L}T)(\lambda)(\mathcal{L}M)(\lambda).\nonumber\\
\end{eqnarray}
Since
 \begin{eqnarray*}
 \lim_{\theta\rightarrow 0}\int_{\mathbb{K}_{\theta t}}\left| (G_{t})(x)\right|^2dx
 & = & \int_{\R^{d}}\left| (G_{t})(x)\right|^2dx\nonumber \\
 & = & C^{*}t^{-\frac{\beta d}{\alpha}}\nonumber\\
\end{eqnarray*}
 where the second equality follows from Lemma\ref{Lem:Green1}. On the other hand we have that
  $$(\mathcal{L}T)(0) = \int_0^\infty dt\int_{\mathbb{K}_{\theta t}}[G_t(z)]^2dz.$$
  Since $$\int_{0}^{\infty}C^{*}t^{-\frac{\beta d}{\alpha}}dt=\infty,  $$ then we obtain that
  $$
  \lim_{\theta\rightarrow 0}(\mathcal{L}T)(0)= \infty.
  $$

Therefore, there exists $\theta_0>0$ such that $(\mathcal{L}T)(0)>L_\sigma^{-2}$ whenever $\theta\in(0,\theta_0).$ This and dominated convergence theorem together imply that there, in turn, will exist $\lambda_0>0$ such that $(\mathcal{L}T)(\lambda)> L_\sigma^{-2}$ whenever $\theta\in(0,\theta_0)$ and $\lambda\in(0,\lambda_0).$
 Since $u_0>0$ on a set of positive measure, it follows readily that $$\int_0^\infty e^{-\lambda t}dt\int_{\mathbb{K}^{+}_{\theta t}} dx |(G_t\ast u_0)(x)|^2 > 0,$$
for all $\theta, \lambda\geq 0,$ including $\theta\in(0,\theta_0)$ and $\lambda\in(0,\lambda_0).$ Therefore, (\ref{Eq200}) implies that
\begin{equation}
(\mathcal{L}M)(\lambda) = \infty\ \ \ \mbox{for}\ \theta\in(0,\theta_0)\ \mbox{and}\ \lambda\in(0,\lambda_0).
\end{equation}
Combining this with the fact that
$$ \int_{|y|>\theta t}\E(|u_t(y)|^2)dy\geq \int_{\mathbb{K}_{\theta t}} \E(|u_t(y)|^2)dy =  M(t) ,$$
one can deduce from this and the definition of $M$ that
$$\limsup_{t\rightarrow\infty}e^{-\lambda t}\int_{|y|>\theta t}\E(|u_t(y)|^2)dy = \infty,$$
whenever $\theta\in(0,\theta_0)\ \mbox{and}\ \lambda\in(0,\lambda_0).$ This and the already-proven first part \eqref{intermittencypart1} together show that
$$\limsup_{t\rightarrow\infty}e^{-\lambda t}\int_{\theta t <|y|< \gamma t}\E(|u_t(y)|^2)dy = \infty,$$
whenever $\theta\in(0,\theta_0), \lambda\in(0,\lambda_0) \ \mbox{and}\ \gamma > (2\nu)^{1/\beta}(\mbox{Lip}_\sigma c_0)^{2(\frac{2-\beta}{2-\beta d})}.$ Since the last integral is not greater than $(\gamma - \theta)t\sup_{|x|>\theta t}\E(|u_t(x)|^2),$ it follows that
$$\mathscr{L}(\theta)=\limsup_{t\rightarrow\infty}\frac{1}{t}\sup_{|x|>\theta t}\log \E(|u_t(x)|^2) \geq \lambda_0,$$
for $\theta\in(0,\theta_0).$ This proves \eqref{intermittencypart2} and hence the theorem.

\end{proof}





\end{document}